\documentclass[12pt,reqno,sort]{elsarticle}
\usepackage{amsfonts,color,amsmath,amssymb,fancyhdr}
\usepackage{amsthm}
\usepackage{graphicx}

\usepackage{subfigure}
\usepackage{enumitem}

\def\Box{\vcenter{\vbox{\hrule\hbox{\vrule
     \vbox to 8.8pt{\hbox to 10pt{}\vfill}\vrule}\hrule}}}

 \newcommand{\bZ}{\mathbb Z}
\newcommand{\bQ}{\mathbb Q} \newcommand{\bC}{\mathbb C}
\newcommand{\bF}{\mathbb F}

 \newcommand{\cN}{\mathcal N}
\newcommand{\cP}{\mathcal P} \newcommand{\cS}{\mathcal S}

\def\Tr{{\rm Tr}}
\def\la{{\langle}} \def\ra{{\rangle}}

\newtheorem{thm}{Theorem}[section]
\newtheorem{cor}{Corollary}[section]

\newtheorem{definition}{Definition}[section]
\newtheorem{remark}{Remark}[section]
\newtheorem{proposition}{Proposition}[section]

\begin{document}
\newcommand{\stopthm}{\begin{flushright}
		\(\box \;\;\;\;\;\;\;\;\;\; \)
\end{flushright}}

\newcommand{\symfont}{\fam \mathfam}

\title{On Constructions and Enumeration of Vectorial Hyper-bent Functions in the $\cP\cS_{ap}^{\#}$ Class}

\date{}
\author[add1]{Jingkun Zhou\corref{cor1}}\ead{jingkunz@zju.edu.cn}\cortext[cor1]{Corresponding author}
\address[add1]{School of Mathematical Sciences, Zhejiang University, Hangzhou 310027, Zhejiang, P.R. China}
\author[add2]{Chunming Tang}\ead{tangchunmingmath@163.com}
\address[add2]{School of Information Science and Technology, Southwest Jiaotong University, Chengdu 610031, China}
\author[add3,add4]{Fengrong Zhang}\ead{zhfl203@163.com}
\address[add3]{State Key Laboratory of Integrated Services Networks, Xidian University, Xian 710071, People's Republic of China}
\address[add4]{School of Computer Science and Technology, China University of Mining and Technology, Xuzhou 221116, Jiangsu, People's Republic of China}
\begin{abstract}

The purpose of this paper is to give explicit constructions of vectorial hyper-bent functions in the $\cP\cS_{ap}^{\#}$ class. It seems that the explicit
constructions were so far known only for very special cases.  To this end, we
present a sufficient and necessary condition of this family of vectorial functions to be hyper-bent.  The conditions are expressed in terms of group ring. Using this characterization, explicit constructions of vectorial hyper-bent functions of the $\cP\cS_{ap}^{\#}$ class
via balanced functions are proposed.  Furthermore,
exact number of vectorial hyper-bent functions in the $\cP\cS_{ap}^{\#}$ class is found.
  The results improve some previous work. Moreover, we solve a problem of counting vectorial hyper-bent functions left by Muratovi\'c-Ribi\'c, Pasalic and Ribi\'c in [{\em IEEE Trans. Inform. Theory}, 60 (2014), pp. 4408-4413].
	\newline
	
	\noindent\text{Keywords:}  Boolean function, Bent function, Hyper-bent function, Vectorial function, Maximum nonlinearity.
	
\end{abstract}	

\maketitle

\section{Introduction}

A hyper-bent function, firstly introduced by A.M.Youssef and G. Gong \cite{Youssef2001} in 2001, is a Boolean function $f:\;\bF_{2^n}\rightarrow\bF_2$ such that its extended Walsh-Hadamard transform
$$\widehat{\chi}_f(\lambda,t)=\sum_{x\in\bF_{2^n}}(-1)^{f(x)+\Tr_1^n(\lambda x^t)}$$
only take the values $\pm2^{\frac{n}{2}}$, where $\lambda\in\bF_{2^n}$ and $t$ is an integer coprime with $2^n-1$. Hyper-bent functions are defined as a special class of bent functions for the purpose of avoiding approximation by a bijective monomial function. The class of hyper-bent functions proposed in \cite{Youssef2001} belong to the $\cP\cS_{ap}$ class of bent functions introduced by Dillon \cite{Dillon1974} and is the only known infinite class of hyper-bent functions up to now.

The hyper-bent property of Boolean functions can be extended to vectorial functions. For a vectorial function $F:\;\bF_{2^n}\rightarrow\bF_2^k$, we say $F$ is hyper-bent if all nonzero combinations of the component functions of $F$ are hyper-bent. That is, a vectorial function $$F(x)=(f_1(x),\dots,f_k(x))$$
is called hyper-bent if $a_1f_1(x)+\cdots+a_kf_k(x)$ is a hyper-bent function for any choice of $a_i\in\bF_2$, where not all of the $a_i$'s are zero.

The trace functions are useful tools for the study of bent functions. Carlet and Gaborit \cite{Carlet2006} have shown that hyper-bent Boolean functions of the $\cP\cS_{ap}^{\#}$ class are of the form $$f(x)=\Tr^{2m}_1\left (\sum_{i=1}^{2^m}a_ix^{i(2^m-1)}+a_0 \right ),$$
where $a_i\in\bF_{2^n}$. Charpin and Gong \cite{Charpin2008} gave a characterization of hyper-bent Boolean functions on $\bF_{2^{2m}}$ of the form $\sum_{r\in R}\Tr_1^{2m}(a_rx^{r(2^m-1)})$ in terms of Dickson polynomials and Kloosterman sums, where $a_r \in \bF_{2^m}$. See paper \cite{Mesnager2013},\cite{Flori2013} for recent progress on hyper-bent Boolean functions with Dillon-like exponents. By employing the M\"obius transformation, Carlet et al.
presented a characterization for the hyper-bentness property of functions with Dillon-like exponents with coefficients in the whole $\bF_{2^{2m}}$ \cite{CHMT}.

In \cite{Pasalic,Lapierre2016} some necessary conditions of vectorial bent functions with Dillon-like exponents are given. In \cite{Muratovic2014}, the authors considered the vectorial case and obtained a sufficient and necessary condition for the function $F(x)=\Tr_m^{2m}(\sum_{i=1}^{2^m}a_ix^{i(2^m-1)})$ to be a vectorial bent function. They also showed that each vectorial bent function of this form is vectorial hyper-bent. Besides, the authors counted the exact number of vectorial hyper-bent functions of the form $F(x)=\Tr_m^{2m}(\sum_{i=1}^{2^m}a_ix^{i(2^m-1)}+c)$.  They also left the question about the cardinality of
vectorial hyper-bent functions for a general case $k|n$ as an
open problem.

In this paper, we study the hyper-bent property of a family of vectorial functions $F:\;\bF_{2^{2m}}\rightarrow\bF_2^k$ such that $f(\gamma^{2^m+1} x)=f(x)$ hold for each nonzero combination $f$ of the component functions of $F$ and $F(0)=0$. The notion of vectorial hyper-bent functions under this condition is a generalization of hyper-bent Boolean functions of the $\cP\cS_{ap}^{\#}$ class. We attain a sufficient and necessary condition of this family of vectorial functions to be hyper-bent. A numerical result for vectorial hyper-bent functions of a typical form is given, which solve the open problem on counting vectorial hyper-bent functions in \cite{Muratovic2014},and a construction of vectorial hyper-bent functions of arbitrary dimension is obtained.

The rest of this paper is organized as follows. In Section 2, we present some preliminaries on group rings and vectorial hyper-bent functions. In Section 3, we establish a sufficient and necessary condition for a class of vectorial functions to be hyper-bent. The number of vectorial hyper-bent functions of a typical form is counted and an explicit construction of vectorial hyper-bent functions is given by using balanced functions. And finally Section 4 concludes the paper.

\section{Preliminaries}
\subsection{Group ring and Fourier analysis}
In this section we introduce some basic results on group rings.
\begin{definition}
Let $G$ be an abelian group. The group ring $\bQ[G]$ is defined as the set of the formal sums of elements of $G$ with coefficients in $\bQ$. The addition, the scalar multiplication and the multiplication in $\bQ[G]$ are respectively defined as follows:
$$\begin{aligned}\sum_{g\in G}a_g\cdot g+\sum_{g\in G}b_g\cdot g&=\sum_{g\in G}(a_g+b_g)\cdot g,\\
                     a\sum_{g\in G}a_g\cdot g&=\sum_{g\in G}(aa_g)\cdot g,
\end{aligned}$$
and
$$\quad \left (\sum_{g\in G}a_g\cdot g \right )\cdot \left (\sum_{g\in G}b_g\cdot g \right )=\sum_{g\in G} \left (\sum_{h\in G}a_hb_{gh^{-1}} \right )\cdot g.$$
\end{definition}

It becomes conventional to abuse the notation $S$ as a subset of $G$ and the corresponding element $\sum_{s\in S}s$ in $\bQ[G]$ at the same time. A character $\chi$ of an abelian group $G$ is a group homomorphism from $G$ to the multiplicative group of the complex field $\bC$. Denote by $\hat{G}$ the set of all the characters of $G$. Let $A=\sum_{g\in G}a_g\cdot g\in\bQ[G]$. The character sum of $\chi$ on $A$ is $\chi(A)=\sum_{g\in G}a_g\chi(g)$. The following inversion formula tells that two elements in $\bQ[G]$ coincide if the values of character sums on them equal for each character.
\begin{proposition}\label{inv for}
Suppose $A=\sum_{g\in G}a_g g$ is an element of the group ring $\bQ[G]$ for a finite abelian group $G$, then the coefficients $a_g$'s of $A$ can be computed explicitly by
$$a_g=\frac{1}{|G|}\sum_{\chi\in\hat{G}}\chi(A)\chi(g^{-1}),$$
where $\hat{G}$ denotes the character group of $G$. In particular, if $A,B\in\bQ[G]$ satisfy $\chi(A)=\chi(B)$ for all characters $\chi\in\hat{G}$, then $A=B$.
\end{proposition}
Since the field $\bF_{2^k}$ are identical to $\bF_2^k$ as a vector space over $\bF_2$, there are two main approaches for describing the set of all characters of an elementary abelian group, one using the dot product and the other using the trace function.
\begin{proposition}\label{char}
Let $k$ be a positive integer.
\\(1) For each $a=(a_1,\dots,a_k)\in\bF_2^k$, define the function $\chi_a:\;\bF_2^k\rightarrow\{\pm1\}$ by
$$\chi_a(x)=(-1)^{\sum_{i=1}a_ix_i}=(-1)^{\la a,x\ra}$$
for each $x=(x_1,\dots,x_k)\in\bF_2^k$, where $\la a,x\ra$ is the usual dot product. Then
$$\widehat{\bF_2^k}=\{\chi_a:a\in\bF_2^k\}.$$
\\(2) For each $a\in\bF_{2^k}$, define the function $\rho_a:\bF_{2^k}\rightarrow\{\pm1\}$ by
$$\rho_a(x)=(-1)^{\Tr_1^k(ax)}$$
for each $x\in\bF_{2^k}$. Then the set $\{\rho_a:a\in\bF_{2^k}\}$ comprises all the characters of $G$, where $G$ is the additive group of $\bF_{2^k}$.
\end{proposition}

\subsection{Hyper-bent functions}
We establish two useful propositions which will be utilized in Section 3. Suppose that $n=2m$ is an even positive integer. We have a straightforward partition of $\bF_{2^n}^*$ as follows:
$$\bF_{2^n}^*=\bigcup_{u\in U}u\bF_{2^m}^*,$$
where $U$ is the cyclic subgroup of $\bF_{2^n}^*$ of order $2^m+1$.

The following proposition is well known, and it can be found in a
slightly different form in \cite{Mesnager}. We give its proof here for the sake of completeness.
\begin{proposition}\label{psap}
Let $\gamma$ be a primitive element of $\bF_{2^n}$. Let $f$ be a Boolean function defined on $\bF_{2^n}$ such that
\begin{equation}\label{equ1}
f(\gamma^{2^m+1}x)=f(x)
\end{equation}
for every $x\in\bF_{2^n}$ and $f(0)=0$. Then $f$ is a hyper-bent function if and only if
$$\sum_{u\in U}(-1)^{f(u)}=1,$$
In this case $f$ is said to belong to the $\cP\cS_{ap}^{\#}$ class.
\end{proposition}
\begin{proof} Denote $q=2^m$. Since $\gamma^{q+1}$ is a primitive element of $\bF_q$, by (\ref{equ1}) we see that for each $u\in U$, the restriction of the function $f$ to $u\bF_q^*$ is constant.

Let $i$ be an integer coprime with $q^2-1$ and $a\in\bF_{q^2}^*$. Then we have
$$\sum_{x\in\bF_{q^2}^*}(-1)^{f(x)+\Tr^n_1(ax^i)}=\sum_{u\in U}(-1)^{f(u)}\sum_{x\in u\bF_q^*}(-1)^{\Tr^n_1(ax^i)}=\sum_{u\in U}(-1)^{f(u)}\sum_{y\in\bF_q^*}(-1)^{\Tr^n_1(au^iy^i)}.$$
For $a\in\bF_{q^2}^*,\;u\in U$ and $y\in\bF_q^*$, we have $$\Tr^n_1(au^iy^i)=\Tr^m_1(\Tr^n_m(au^iy^i))=\Tr^m_1(au^iy^i+a^qu^{qi}y^{qi}),$$
by transitivity of trace.
Since $u^q=u^{-1}$ and $y^q=y$, we have
$$\Tr^n_1(au^iy^i)=\Tr^m_1(au^iy^i+a^qu^{qi}y^{qi})=\Tr^m_1((au^i+a^qu^{-i})y^i).$$
Then $au^i+a^qu^{-i}=0$ if and only if $u=a^{t(q-1)}$, where $t$ is the multiplicative inverse of $2i$ modulo $q^2-1$.

Set $u_0=a^{t(q-1)}$. If $u=u_0$, then for any $y\in\bF_q^*$, we have $\Tr^n_1(au^iy^i)=0$, and hence
$$\sum_{y\in\bF_q^*}(-1)^{\Tr^n_1(au^iy^i)}=q-1.$$
If $u\in U\setminus\{u_0\}$, we have $au^i+a^qu^{-i}\ne0$. Then the set
$$\{y\in\bF_q^*:\Tr^m_1((au^i+a^qu^{-i})y^i)=0\}$$
has size $q/2-1$ and the set
$$\{y\in\bF_q^*:\Tr^m_1((au^i+a^qu^{-i})y^i)=1\}$$
has size $q/2$ as $y\mapsto y^i$ is a bijection. So we deduce that
$$\sum_{y\in\bF_q^*}(-1)^{\Tr^n_1(au^iy^i)}=1\times(q/2-1)+(-1)\times q/2=-1$$
for each $u\in U\setminus\{u_0\}$. Hence
$$\begin{aligned}\sum_{x\in\bF_{q^2}^*}(-1)^{f(x)+\Tr^n_1(ax^i)}&=\sum_{u\in U}(-1)^{f(u)}\sum_{y\in\bF_q^*}(-1)^{\Tr^n_1(au^iy^i)}\\
&=(q-1)\cdot(-1)^{f(u_0)}-\sum_{u\in U\setminus\{u_0\}}(-1)^{f(u)}\\
&=q\cdot(-1)^{f(u_0)}-\sum_{u\in U}(-1)^{f(u)}\\
&=\pm q-\sum_{u\in U}(-1)^{f(u)}.
\end{aligned}$$
Therefore, we have
$$\sum_{x\in\bF_{q^2}}(-1)^{f(x)+\Tr^n_1(ax^i)}=\pm q+1-\sum_{u\in U}(-1)^{f(u)}.$$
Then by the definition of a hyper-bent function we conclude that $f$ is a hyper-bent function if and only if
$$\sum_{u\in U}(-1)^{f(u)}=1,$$
which completes the proof of the proposition.
\end{proof}

\begin{proposition}\label{prop fg}
Let $f$ be a hyper-bent function in $\cP\cS_{ap}^{\#}$ with $f(0)=0$. Then there exists a function $g$ from $U$ to $\bF_2$ such that
$$f(x)=g(x^{2^m-1})$$
for $x\in\bF_{2^n}^*$, and
$$\sum_{u\in U}(-1)^{g(u)}=1.$$
Conversely, if $g$ is a function from $U$ to $\bF_2$ such that $\sum_{u\in U}(-1)^{g(u)}=1$, then the function $f$ from $\bF_{2^n}$ to $\bF_2$ such that $f(x)=g(x^{2^m-1})$ for $x\in\bF_{2^n}^*$ and $f(0)=0$ is a hyper-bent function from $\cP\cS_{ap}^{\#}$.
\end{proposition}
\begin{proof}
First assume that $f$ is a hyper-bent function in $\cP\cS_{ap}^{\#}$ with $f(0)=0$. Take $s$ to be the multiplicative inverse of $2^m-1$ modulo $2^m+1$ and define $g:\;U\rightarrow\bF_2$ such that $g(u)=f(u^s)$ for each $u\in U$. For $x\in\bF_{2^n}^*$, take $u\in U$ and $y\in\bF_q^*$ such that $x=uy$. Then we have
$$f(x)=f(u)=f(u^{s(2^m-1)})=f((u^{2^m-1})^s)=g(u^{2^m-1})=g(x^{2^m-1}).$$
And we see that
$$\sum_{u\in U}(-1)^{g(u)}=\sum_{u\in U}(-1)^{f(u^s)}=\sum_{u\in U}(-1)^{f(u)}=1$$
by Proposition \ref{psap} and $\gcd(2^m+1,s)=1$.

Now assume that $g$ is a function from $U$ to $\bF_2$ such that $\sum_{u\in U}(-1)^{g(u)}=1$, and that $f$ is a Boolean function of $\bF_{2^n}$ such that $f(x)=g(x^{2^m-1})$ for $x\in\bF_{2^n}^*$ and $f(0)=0$. Notice that
$$f(\alpha^{q+1}x)=f(x)$$
where $\alpha$ is a primitive element of $\bF_{2^n}$, and that
$$\sum_{u\in U}(-1)^{f(u)}=\sum_{u\in U}(-1)^{g(u^{2^m-1})}=\sum_{u\in U}(-1)^{g(u)}=1$$
as $\gcd(2^m-1,2^m+1)=1$. By Proposition \ref{psap}, $f$ is a hyper-bent function from $\cP\cS_{ap}^{\#}$. The proof is now complete.
\end{proof}
\section{Constructions and enumeration of vectorial hyper-bent functions}

Let $n,m,\gamma$ and $U$ be as in Section 2.2. We have a the following characterization of vectorial hyper-bent functions in the context of group rings.
\begin{thm}\label{vhf} Let $n=2m$. Let $F(x)$ be a vectorial function from $\bF_{2^n}$ to $\bF_2^k$ such that $f(\gamma^ {2^m+1} x)=f(x)$ hold for each nonzero combination $f$ of the component functions of $F$ and $F(0)=0$. Then the following conditions are equivalent:
\\(1) $F$ is a vectorial hyper-bent function of dimension $k$.
\\(2) $\sum_{u\in U}(-1)^{\la v,F(u)\ra}=1$ for all $v\in\bF_2^k\setminus\{0\}$.
\\(3) $\sum_{u\in U}F(u)=2^{m-k}H+0_H$ holds in the group ring $\bZ[H]$, where $H$ is the additive group of $\bF_2^k$.
\end{thm}
\begin{proof}
By the definition of a vectorial hyper-bent function, we deduced that $F$ is a vectorial hyper-bent function if and only if $\la v,F(u)\ra$ is bent for all $v\in\bF_2^k\setminus\{0\}$, which is equivalent to (2) by Proposition \ref{psap}. Hence (1) and (2) are equivalent.

Suppose $\sum_{u\in U}F(u)=2^{m-k}H+0_H$ holds. Taking an arbitrary non-principal character $\chi_v$ of $H$, we see that
$$\sum_{u\in U}(-1)^{\la v,F(u)\ra}=\sum_{u\in U}\chi_v(F(u))=\chi_v(2^{m-k}H+0_H)=1,$$
where $\chi_v(x)=(-1)^{\la v,x\ra}$.

Suppose $\sum_{u\in U}(-1)^{\la v,F(u)\ra}=1$ for all $v\in\bF_2^k\setminus\{0\}$. Then $\chi_v(\sum_{u\in U}F(u))=\chi_v(2^{m-k}H+0_H)=1$ for each non-principal character $\chi_v$. Also we have $\psi(\sum_{u\in U}F(u)=\psi(2^{m-k}H+0_H)=2^m+1$ where $\psi$ denotes the principal of $H$. Then by Proposition \ref{inv for} we conclude that (3) holds. Thus (2) and (3) are equivalent and the proof is now complete.
\end{proof}
Since each hyper-bent Boolean function of $\bF_{2^n}$ from $\cP\cS_{ap}^{\#}$ are of the form
$$f(x)=\Tr^n_1\left (\sum_{i=1}^{2^m}a_ix^{i(2^m-1)}+a_0 \right ),$$
any vectorial hyper-bent function from $\bF_{2^n}$ to $\bF_2^k$ satisfying the condition given in Theorem \ref{vhf}  has the following expression:
$$F(x)=\left (\Tr_1^n\left (\sum_{i=1}^{2^m}a_{1,i}x^{i(2^m-1)} +a_{1,0} \right ),\dots,\Tr_1^n \left (\sum_{i=1}^{2^m}a_{k,i}x^{i(2^m-1)} + a_{k,0}\right ) \right ).$$
Our next theorem counts the number of vectorial hyper-bent functions of this form.

\begin{thm}
Let $n=2m$. Let $\cN_{n,k}$ denote the number of vectorial hyper-bent functions of the form
$$F(x)=\left (\Tr_1^n\left (\sum_{i=1}^{2^m}a_{1,i}x^{i(2^m-1)} +a_{1,0} \right ),\dots,\Tr_1^n \left (\sum_{i=1}^{2^m}a_{k,i}x^{i(2^m-1)} + a_{k,0}\right ) \right ).$$
Then,
$$\cN_{n,k}=2^k\cdot\binom{2^m+1}{2^{m-k}+1}\cdot\prod_{i=1}^{2^k-1}\binom{2^m-i\cdot2^{m-k}}{2^{m-k}}.$$
\end{thm}
\begin{proof}
We first consider the number $N$ of vectorial hyper-bent functions $F$ of the given form such that $F(0)=0$. By Theorem \ref{vhf} we deduce that
$$N=\binom{2^m+1}{2^{m-k}+1}\cdot\prod_{i=1}^{2^k-1}\binom{2^m-i\cdot2^{m-k}}{2^{m-k}}.$$
Since each vectorial hyper-bent function of the given form is a translation $F(x)+r,\;r\in\bF_{2^k}$, of a vectorial hyper-bent function $F(x)$ of the given form such that $F(0)=0$ and each such $F(x)$ has exactly $2^k$ translations, we obtain the desired result.
\end{proof}

Let $k$ be a positive integer. Given an ordered basis $A=(\alpha_1, \dots, \alpha_k)$ of
$\bF_{2^k}$ over $\bF_2$,  its dual basis is defined to be
a basis
$B=(\beta_1,\dots, \beta_k)$
satisfying
$$\mathrm{Tr}^k_1( \alpha_i \beta_j)=\delta_{i,j} \text{ for } i, j=1,2, \dots, k,$$
where $\delta_{i,j}$ denotes the Kronecker delta function. It is well known that each basis of $\bF_{2^k}$ over $\bF_2$ has a unique dual basis.

Let $n=2m$ and $k$ be an integer with $k|m$. Let us denote by $\mathcal{HB}_{n,k}$
the set of all the hyper-bent functions from $\bF_{2^n}$ to $\bF_{2}^k$, of the form
$$F(x)=\left (\Tr_1^n\left (\sum_{i=1}^{2^m}a_{1,i}x^{i(2^m-1)} +a_{1,0} \right ),\dots,\Tr_1^n \left (\sum_{i=1}^{2^m}a_{k,i}x^{i(2^m-1)} + a_{k,0}\right ) \right ),$$
where $a_{i,j} \in \bF_{2^n}$ for $i\in \{1, \dots,k\}$ and $j \in \{0, \dots, 2^m\}$.
Let $\widetilde{\mathcal{HB}}_{n,k}$ denote the set of all the hyper-bent functions from $\bF_{2^n}$ to $\bF_{2^k}$, of the form
$$\widetilde{F}(x)= \Tr_k^n\left (\sum_{i=1}^{2^m}b_{i}x^{i(2^m-1)} +b_{0} \right ),$$
where $b_{i} \in \bF_{2^n}$ for  $i \in \{0, \dots, 2^m\}$.
Let $A=(\alpha_1, \dots, \alpha_k)$
be a basis of $\bF_{2^k}$ over $\bF_2$ and
$B=(\beta_1,\dots,\beta_k)$
its dual basis.
We define a mapping $\pi$ from $\mathcal{HB}_{n,k}$ to $\widetilde{\mathcal{HB}}_{n,k}$ as follows: $\pi (f_1(x), \dots, f_k(x))=  \sum_{j=1}^k f_j(x) \alpha_j $, where
$(f_1(x), \dots, f_k(x)) \in \mathcal{HB}_{n,k}$.
And define a mapping $\sigma$ from $\widetilde{\mathcal{HB}}_{n,k}$  to $\mathcal{HB}_{n,k}$ as follows: $\sigma (\widetilde{F}(x))=
\left ( \mathrm{Tr}^k_1 (\beta_1 \widetilde{F}(x)), \dots, \mathrm{Tr}^k_1 (\beta_k \widetilde{F}(x)) \right )$, where
$\widetilde{F}(x) \in \widetilde{\mathcal{HB}}_{n,k}$. A trivial verification shows that
$\pi \sigma (\widetilde{F}(x))= \widetilde{F}(x)$ and $ \sigma\pi (f_1(x), \dots,$ $
 f_k(x))$
$=(f_1(x), \dots, f_k(x))$ for $\widetilde{F}(x) \in \widetilde{\mathcal{HB}}_{n,k}$ and
$(f_1(x), \dots, f_k(x)) \in \mathcal{HB}_{n,k}$. Consequently, the hyper-bent functions from $\widetilde{\mathcal{HB}}_{n,k}$ are exactly those
elements of $\mathcal{HB}_{n,k}$.

\begin{remark}

Let $n=2m$. The number of vectorial hyper-bent functions in $\widetilde{\mathcal{HB}}_{n,m}$ is counted in \cite{Muratovic2014}, so our result is a generalization of the case $k=m$.
\end{remark}

\begin{thm}\label{thm3}
Let $n=2m$ and $u_0\in U\setminus \{1\}$. Let $T_{u_0}$ be the vectorial function
defined on $\bF_{2^n}$ by
$$T_{u_0}(x)=\Tr_m^n \left (u_0\sum_{i=1}^{2^{m-1}}x^{i(2^m-1)} \right ).$$
Then $T_{u_0}$ is a vectorial hyper-bent function.
\end{thm}

\begin{proof}
Denote
$$g(u)=\Tr_m^n \left (u_0\sum_{i=1}^{2^{m-1}}u^i \right )\ \text{for}\ u\in U$$
and notice that $g(1)=0$. By Proposition \ref{prop fg}, it suffices to show that $g|_{U\setminus\{1\}}$ maps $U\setminus\{1\}$ onto $\bF_{2^m}$.
Suppose that there exist $u_1\ne u_2\in U\setminus\{1\}$ such that
\begin{equation}\label{3.3.1}
\Tr_m^n \left (u_0\sum_{i=1}^{2^{m-1}}u_1^i \right )=\Tr_m^n \left (u_0\sum_{i=1}^{2^{m-1}}u_2^i \right ).
\end{equation}
Now that
$$\begin{aligned}
\Tr_m^n\left (u_0\sum_{i=1}^{2^{m-1}}u_1^i \right )&=\Tr_m^n \left(u_0u_1\frac{1+u_1^{2^{m-1}}}{1+u_1} \right )\\
&=u_0u_1\frac{1+u_1^{2^{m-1}}}{1+u_1}+u_0^{2^m}u_1^{2^m}\frac{1+u_1^{2^{2m-1}}}{1+u_1^{2^m}}\\
&=u_0u_1\frac{1+u_1^{2^{m-1}}}{1+u_1}+u_0^{-1}\frac{1+u_1^{-2^{m-1}}}{1+u_1},
\end{aligned}$$
we compute that
$$\begin{aligned}
\left (\Tr_m^n \left (u_0\sum_{i=1}^{2^{m-1}}u_1^i \right ) \right )^2&=u_0^2u_1^2\frac{1+u_1^{-1}}{1+u_1^2}
+u_0^{-2}\frac{1+u_1}{1+u_1^2}\\
&=u_0^2\frac{1+u_1^{-1}}{1+u_1^{-2}}+u_0^{-2}\frac{1+u_1}{1+u_1^2}\\
&=\frac{u_0^2}{1+u_1^{-1}}+\frac{u_0^{-2}}{1+u_1}.
\end{aligned}$$
Similarly, we have
$$ \left(\Tr_m^n \left (u_0\sum_{i=1}^{2^{m-1}}u_2^i \right ) \right )^2=\frac{u_0^2}{1+u_2^{-1}}+\frac{u_0^{-2}}{1+u_2}.$$
Then (\ref{3.3.1}) implies that
$$\frac{u_0^2}{1+u_1^{-1}}+\frac{u_0^{-2}}{1+u_1}=\frac{u_0^2}{1+u_2^{-1}}+\frac{u_0^{-2}}{1+u_2},$$
that is,
$$\frac{u_0^2u_1+u_0^{-2}}{1+u_1}=\frac{u_0^2u_2+u_0^{-2}}{1+u_2}.$$
Then
$$(u_0^2u_1+u_0^{-2})(1+u_2)=(u_0^2u_2+u_0^{-2})(1+u_1).$$
It follows that
$$(u_0^2+u_0^{-2})(u_1+u_2)=0.$$
Since $u_0\ne1$, we have $u_0^2\ne u_0^{-2}$. Hence $u_1=u_2$, contradictory. Then $g|_{U\setminus\{1\}}$ is injective, and since $U\setminus\{1\}$ and $\bF_{2^m}$ are of equal size it is also surjective. The proof is now complete.
\end{proof}

A function $h$ from $\bF_{2^m}$ to $\bF_2^k$ is called \emph{balanced} if $\#\{x\in\bF_{2^m}:h(x)=b\}=2^{m-k}$ for any $b\in\bF_2^k$. We have the following straightforward construction of vectorial hyper-bent functions from $\bF_{2^{2m}}$ to $\bF_2^k$.
\begin{thm}\label{bal}
Let $T_{u_0}$ be defined as in Theorem \ref{thm3}. Let $h$ be a balanced function from $\bF_{2^m}$ to $\bF_2^k$ with $h(0)=0$. Then $h(T_{u_0}(x))$ is a vectorial hyper-bent function from $\bF_{2^n}$ to $\bF_2^k$.
\end{thm}
\begin{proof}
It follows directly from Theorem \ref{vhf} and Theorem \ref{thm3}.
\end{proof}

Finally
we present some infinite classes of hyper-bent functions employing permutation polynomials and binary m-sequences.

A polynomial $h(X) \in \bF_{2^m}[X]$ is called a permutation polynomial (PP) of $\bF_{2^m}$ if the
associated polynomial function $h: x \mapsto f(x)$ from $\bF_{2^m}$ to itself is a permutation of $\bF_{2^m}$. A very important class of polynomials whose
permutation behavior is well understood is the class of Dickson polynomials, which we
will define below. We recall that the $r$-th binary Dickson polynomial $D_{r}(x)\in \bF_{2} [x]$ is defined by
\[D_r(x)=\sum_{i=0}^{\lfloor r/2 \rfloor} \frac{r}{r-i} \binom{r-i}{i} x^{r-2i},\]
where $\lfloor r/2 \rfloor$ denotes the largest integer less than or equal to $r/2$.
For $r=0$, we set $D_0(x) =0$. The first eight
binary Dickson polynomials are
\begin{eqnarray*}
\begin{array}{c}
	D_0(x)=0, ~~D_1(x)=x,~~D_2(x)=x^2,~~D_3(x)=x^3+x,
	~~D_4(x)=x^4,\\
	D_5(x)=x^5+x^3+x,~~D_6(x)=x^6+x^2,~~D_7(x)=x^7+x^5+x.		
\end{array}
\end{eqnarray*}

We write $x =\frac{1}{y}$ with $y\neq 0$ an indeterminate. Then binary Dickson polynomials can often
be rewritten (also referred as functional expression) as
\[D_r(x)=D_r\left (y+\frac{1}{y}\right )=y^r+\frac{1}{y^r}.\]
For any non-zero
positive integers $r$ and $s$, Dickson polynomials satisfy:
\[D_r(D_s(x))=D_{rs}(x).\]

The PPs among the Dickson polynomials have been completely
classified. We state the following theorem due to N\"obauer \cite{Noba}.
Dickson in his 1896 Ph. D. thesis observed and partially proved the theorem.
\begin{thm}\label{Dickson}
The Dickson polynomial $D_r(x)$ is a permutation polynomial
of $\bF_{2^m}$ if and only if $\mathrm{gcd}(r,2^{2m}-1)=1$.
\end{thm}
Combining Theorem \ref{bal} with Theorem \ref{Dickson} gives the following construction of vectorial hyper-bent functions from Dickson polynomials.
\begin{cor}
Let $n=2m$ and let $r$ be a positive integer such that $\mathrm{gcd}(r,2^{2m}-1)=1$. Let $T_{u_0}$ be defined as in Theorem \ref{thm3} and let $D_r(x)$ be the $r$-th binary Dickson polynomial. Then $D_r(T_{u_0}(x))$ is a vectorial hyper-bent function from $\bF_{2^n}$ to $\bF_{2^m}$.
\end{cor}

Any permutation binomial or  permutation trinomial proposed in \cite{LQC} can be plugged into Theorem \ref{bal} to obtain a vectorial hyper-bent function from $\bF_{2^n}$ to $\bF_{2^m}$.

The binary m-sequences of period $2^m-1$ are the sequences of elements in $\bF_2$ of the form
$\left \{\mathrm{Tr}^{m}_1 \left ( \gamma^{di+t} \right ) \right \}_{i \in \mathbb Z}$
where $\gamma$ is a generator of $\bF_{2^m}^*$, $t$ is an integer, and the {\em decimation} $d$ has $\mathrm{gcd}(d,2^m-1)=1$. The crosscorrelation $C_d(t)$ between the two m-sequences $\left \{\mathrm{Tr}^{m}_1 \left ( \gamma^{i} \right ) \right \}_{i \in \mathbb Z}$ and $\left \{\mathrm{Tr}^{m}_1 \left ( \gamma^{di} \right ) \right \}_{i \in \mathbb Z}$
 can be described by the following exponential sum by
using the trace function representation
\begin{eqnarray*}
	\begin{array}{rl}
		C_d(t)&=\sum_{i=0}^{2^m-2} (-1)^{ \mathrm{Tr} \left( \gamma^{t+i} +\gamma^{di} \right )}\\
		&=\sum_{x\in \bF_{2^m}^*} (-1)^{ \mathrm{Tr} \left ( x^d+cx \right )},
	\end{array}
\end{eqnarray*}
where $c=\gamma^t$.  We say that $C_d(t)$ is $v$-valued to mean that $\# \left \{ C_d(t): t \in \mathbb Z\right \} =v$. It was shown by Katz
\cite{katz} that  $C_d(t)$ always takes on
$-1$ as one of the values if the crosscorrelation function $C_d(t)$ is three-valued. Thus,
we have the following construction of hyper-bent functions from three-valued m-sequences.
\begin{cor}
Let $n = 2m$ and $u_0 \in U \setminus \{1\}$. Let $d$ be an arbitrary positive integer such that $\mathrm{gcd}(d,2^m-1)=1$
and $C_d(t)$ is three-valued.  Then there always exists $\lambda \in \bF_{2^m}^{*}$ such that $\mathrm{Tr}^{m}_1\left (  \left ( \Tr_m^n \left (u_0\sum_{i=1}^{2^{m-1}}x^{i(2^m-1)} \right ) \right )^d\right  ) +\Tr_1^n \left (\lambda u_0\sum_{i=1}^{2^{m-1}}x^{i(2^m-1)} \right )$ is a hyper-bent function.
\end{cor}

For binary m-sequences of length $2^m-1$, the following is a complete list of all
decimations known to give three-valued crosscorrelation. It is a challenging and open
problem to decide whether this list is complete.
\begin{enumerate}[label=(\roman*)]
	\item  Gold \cite{gold}: $d=2^k+1$, $\frac{m}{\mathrm{gcd}(k,m)}$ odd.
	\item Kasami \cite{Kasami}: $d=2^{2k}-2^k+1$, $\frac{m}{\mathrm{gcd}(k,m)}$ odd.
	\item Cusick and Dobbertin \cite{Cusick and Dobbertin}: $d=2^{\frac{m}{2}}+2^{\frac{m+2}{4}}+1$, $m\equiv 2 \pmod 4$.
	\item Cusick and Dobbertin \cite{Cusick and Dobbertin}: $d=2^{\frac{m+2}{2}}+3$, $m\equiv 2 \pmod 4$.
	\item Canteaut, Charpin and Dobbertin \cite{Charpin and Dobbertin}: $d=2^{\frac{m-1}{2}}+3$, $m$ odd.
	\item Dobbertin \cite{Dobbertin}, Hollmann and Xiang \cite{Hollmann and Xiang}:
	\begin{eqnarray*}
	d= \left \{ \begin{array}{lr}
		2^{\frac{m-1}{2}}+2^{\frac{m-1}{4}}-1,& m\equiv 1 \pmod 4 \\
		2^{\frac{m-1}{2}}+2^{\frac{3m-1}{4}}-1,& m\equiv 3 \pmod 4
		\end{array} \right . .
	\end{eqnarray*}
\end{enumerate}

\section{Conclusion}
In this paper we are devoted to deducing a sufficient and necessary condition of vectorial hyper-bent functions which is a generalization for case $k=n/2$ of Theorem 1 in \cite{Muratovic2014}. We also get a numerical result for the number $\cN_{n,k}$ of vectorial hyper-bent functions of the form
$$F(x)=\left (\Tr_1^n \left (\sum_{i=1}^{2^m}a_{1,i}x^{i(2^m-1)} +a_{1,0} \right ),\dots,\Tr_1^n \left (\sum_{i=1}^{2^m}a_{k,i}x^{i(2^m-1)} +a_{k,0} \right ) \right ),$$
which generalizes the result of Theorem 4 in \cite{Muratovic2014}. By Theorem \ref{bal}, the problem of searching for vectorial hyper-bent functions $\bF_{2^n}$ to $\bF_2^k$ is reduced to the one of finding balanced functions from $\bF_{2^n}$ to $\bF_2^k$.

\noindent\textbf{Acknowledgement.} This work was supported by National Natural Science Foundation of China under Grant Nos. 12171428, 12231015 and 61972400.

\begin{center}
	\scriptsize
	\setlength{\bibsep}{0.5ex}  
		\linespread{0.5}
	\bibliographystyle{plain}

\end{center}
\end{document}